\documentclass[10pt,reqno]{amsart}
\usepackage{amssymb,mathrsfs}
\usepackage[utf8]{inputenc}
\usepackage[T1]{fontenc}
\usepackage{lmodern}
\usepackage{hyperref}
\usepackage{colortbl}
\usepackage{enumitem}

\def\bC {\mathbb{C}}
\def\bN {\mathbb{N}}
\def\R{\mathbb{R}}

\def\e{\mathrm{e}}

\def\cA {\mathcal{A}}

\def\cD {\mathcal{D}}

\def\cF {\mathcal{F}}

\def\cS {\mathcal{S}}

\def\scrL{\mathscr{L}}

\def\a {{\underline{\alpha}}}

\def\d {{\rm d}}

\newcommand{ \Sgn}{\operatorname{sgn}}
\newcommand{\Supp}{\operatorname{supp}}
\newcommand{\Ker}{\operatorname{ker}}
\newcommand{\rg}{\operatorname{rg}}

\newcommand{\supp}{\operatorname{supp}}
\renewcommand{\Re}{\operatorname{Re}}
\newcommand{\ep}{\varepsilon}

\newcommand{\ba}{\begin{aligned}}
\newcommand{\ea}{\end{aligned}}

\newcommand{\be}{\begin{equation}}
\newcommand{\ee}{\end{equation}}

\newtheorem{theorem}{Theorem}[section]
\newtheorem{corollary}[theorem]{Corollary}

\newtheorem{lemma}[theorem]{Lemma}
\newtheorem{proposition}[theorem]{Proposition}

\theoremstyle{definition}

\newtheorem{remark}{Remark}

\begin{document}

\title[Growth-fragmentation with unbounded fragmentation rate]{Asynchronous exponential growth of the growth-fragmentation equation\\ with unbounded fragmentation rate}

\author[\'E. Bernard]{\'{E}tienne Bernard}
\address[\'E. Bernard]{Université Paris-Est, CERMICS (ENPC), INRIA, 77455 Marne-la-Vallée, France}
\email{etienne.bernard@enpc.fr}
\author[P. Gabriel]{Pierre Gabriel}
\address[P. Gabriel]{Laboratoire de Math\'ematiques de Versailles, UVSQ, CNRS, Universit\'e Paris-Saclay,  45 Avenue des \'Etats-Unis, 78035 Versailles cedex, France.}
\email[Corresponding author]{pierre.gabriel@uvsq.fr}
\thanks{The second author was supported by the ANR project KIBORD, ANR-13-BS01-0004, funded by the French Ministry of Research.}


\begin{abstract}
The objective is to prove the \emph{asynchronous exponential growth} of the growth-fragmentation equation in large weighted $L^1$ spaces and under general assumptions on the coefficients.
The key argument is the creation of moments for the solutions to the Cauchy problem,
resulting from the unboundedness of the total fragmentation rate.
It allows us to prove the quasi-compactness of the associated (rescaled) semigroup, which in turn provides the exponential convergence toward the projector on the Perron eigenfunction.
\end{abstract}



\maketitle


\section{Introduction and main results}

In this article, we study the asymptotic behavior of the {\it growth-fragmentation equation} 
\begin{equation}
\label{eq:GF}
\left\{
    \begin{array}{lll}
    \partial_{t}f(t,x)+\partial_x\left(\tau(x) f(t,x)\right) = \cF f(t,x),&\qquad\qquad&  t,x>0,
    \vspace{2mm}\\
    f(t,0)=0& &t>0,
    \vspace{2mm}\\
    f(0,x) = f^{\rm{in}}(x)& & x>0.
    \end{array}
\right.
\end{equation}
This equation appears in the modeling of various physical or biological phenomena~\cite{Banasiak,MD86,BP,RTK}
as well as in telecommunication.
The unknown $f(t,x)$ represents the concentration at time $t$ of some ``particles'' with ``size'' $x>0,$
which can be for instance the volume of a cell~\cite{DHT}, the length of a fibrillar polymer~\cite{EPW}, the window size in data transmission over the Internet~\cite{ChafaiMalrieuParoux}, or the time elapsed since the last discharge of a neuron~\cite{PPS}.
Each particle grows with a rate $\tau(x)$ and splits according to the fragmentation operator $\cF$ which acts on a function $f(x)$ through
\[\cF f(x):=\cF_+f(x)-B(x)f(x).\]
The positive part $\cF_+$ is an integral operator given by
\be
\label{def:cF+}
\cF_+ f(x) := \int_0^1B\Bigl(\frac xz\Bigr)f\Bigl(\frac xz\Bigr)\frac{\wp(\d z)}{z}.
\ee
When a particle of size $x$ breaks with rate $B(x),$ it produces smaller particles of sizes $zx$ with $0<z<1$ distributed with respect to the fragmentation kernel $\wp.$

\

All along the paper except in Section~\ref{sec:Osgood}, the coefficients of the model are supposed to verify the following hypotheses:

\begin{itemize}

\item[\bf (H$\tau$)]
The growth rate $\tau:(0,\infty)\to(0,\infty)$ is a $C^1$ function which satisfies
\begin{equation}\label{as:tau0}
\frac1\tau\in L^1(0,1),
\end{equation}
and there exist $\nu_0\leq1$ and $\tau_1\geq\tau_0>0$ such that
\begin{equation}\label{as:tau_infty}
\forall x>0,\qquad \tau_0\mathbf1_{x\geq1} x^{\nu_0}\leq \tau(x)\leq\tau_1\max(1,x).
\end{equation}

\medskip

\item[\bf (H$B$)]
The total fragmentation rate $B:(0,\infty)\to[0,\infty)$ is a continuous function with a connected support
and there exist $\gamma_1\geq\gamma_0>0,$ $B_1\geq B_0>0$ and $x_0>0$ such that
\begin{equation}\label{as:B_infty}
\forall x>0,\qquad B_0\mathbf1_{x\geq x_0}x^{\gamma_0}\leq B(x)\leq B_1\max(1,x^{\gamma_1}).
\end{equation}

\medskip

\item[\bf (H$\wp$)]
The fragmentation kernel $\wp$ is a finite positive measure on the open interval $(0,1)$ such that
\be\label{as:mass_cons}\int_0^1 z\,\wp(\d z)=1.\ee

\end{itemize}

\

For any $\alpha\in\R$ we will use the following notation for the (possibly infinite) $\alpha$-moment of the fragmentation kernel
\[\wp_\alpha:=\int_0^1z^\alpha\wp(\d z),\]
and we define
\[\a:=\inf\{\alpha\in\R,\ \wp_\alpha<+\infty\}.\]
Hypothesis {\bf (H$\wp$)} ensures that $1=\wp_1<\wp_0<+\infty,$ so that $\a\in[-\infty,0],$ and $\alpha\mapsto\wp_\alpha$ is strictly decreasing on $(\a,+\infty).$ 
The zero-moment $\wp_0$ represents the mean number of fragments, and the first moment is related to their mean size:
if a particle of size $x$ breaks, the mean size of the fragments is $\frac{\wp_1}{\wp_0}x.$
Condition~\eqref{as:mass_cons} thus guarantees that the fragmentation operator preserves the total size,
{\it i.e.} the sum of all the sizes of the daughter particles is equal to the size of the mother particle (at the statistical level).

Classical examples of fragmentation kernels are the mitosis kernel $\wp=2\delta_{1/2},$ the asymmetrical division kernels $\wp=\delta_\theta+\delta_{1-\theta}$ with $\theta\in(0,1/2),$
and the power law kernels $\wp(\d z)=(\nu+2)z^\nu\d z$ with $\nu>-1.$
Notice that the power law kernels are physically relevant only for $\nu\leq0$ (see discussion in~\cite[Section~8.2.1]{Banasiak}),
which includes the uniform kernel $\wp(\d z)=2\,\d z.$

\

The long time behavior of the solutions is strongly related to the existence of $(\lambda,G,\phi)$ solution to the following Perron eigenvalue problem:
\be\label{eq:direct_Perron}(\tau G)'+\lambda G+BG=\mathcal F_+ G,\qquad G\geq0,\qquad \int_0^\infty G(x)\,\d x=1,\ee
and the dual problem:
\be\label{eq:dual_Perron}-\tau\phi'+\lambda\phi+B\phi=\mathcal F_+^*\phi,\qquad\phi\geq0,\qquad \int_0^\infty G(x)\phi(x)\,\d x=1,\ee
where
\[\cF_+^*\phi(x):=B(x)\int_0^1\phi(zx)\wp(\d z).\]
When $(\lambda,G,\phi)$ exists and for initial distributions which satisfy
\[\langle f^{\rm in},\phi\rangle:=\int_0^\infty f^{\rm in}(y)\phi(y)\,\d y<+\infty,\]
the solutions to Equation~\eqref{eq:GF} are expected to behave like
\[f(t,x)\sim\langle f^{\rm in},\phi\rangle\,G(x){\rm e}^{\lambda t}\qquad\text{when}\ t\to+\infty.\]
This property is sometimes called {\it asynchronous exponential growth}~\cite{Webb87} since it ensures that the shape of the initial distribution is forgotten for large times.
Asymptotically the population grows exponentially fast with a {\it Malthus parameter} $\lambda$ and is aligned to the {\it stable size distribution} $G.$

\

Asynchronous exponential growth for growth-fragmentation was first proved by Diekmann, Heijmans and Thieme~\cite{DHT}.
In this pioneer paper the size state space is supposed to be bounded, an assumption also made in~\cite{BPR,GreinerNagel,H84,H86b,RP}.
When the size variable lies in $(0,\infty)$ the \emph{General Relative Entropy} introduced in~\cite{MMP2}
allows to prove the asynchronous exponential growth in weighted $L^p$ spaces for fairly general coefficients,
but without rate of convergence.
Obtaining an exponential rate of convergence in the case of an unbounded state space produced a large literature since the result of Perthame and Ryzhik~\cite{PR05}.
Let us review here these existing results, some of which deal with the (slightly simpler) conservative form of the equation when the condition $\wp_1=1$ is replaced by $\wp_0=1$ (in this case $\lambda=0$ and $\phi=1$).

The exponential decay of the $L^1$ norm was obtained by analytical methods (functional inequalities) in~\cite{LP09,PPS,PR05} and probabilistic methods (coupling arguments) in~\cite{BCGMZ,Broda,Malrieu}.
However the convergence is controlled by a distance between the initial distribution and the asymptotic profile which is stronger than the $L^1$ norm.
A spectral gap was proved by means of Poincaré type inequalities in Hilbert spaces~\cite{BCG,CCM11,GS14,Monmarche},
and in weighted $L^1$ spaces by semigroup techniques~\cite{CCM10,MS} and probabilistic methods~\cite{Bertoin19,BW18,Bouguet}.
Let us also mention that another type of convergence than in norm was considered in~\cite{ZvBW2}, where a higher order pointwise asymptotic expansion is provided.
Besides, a spectral gap in weighted $\ell^1$ spaces has been recently proved in~\cite{BJS19} for the discrete growth-fragmentation model.

Convergence in weighted $L^1$ spaces is of particular interest.
First, weighted $L^1$ norms have physical interpretation: for instance the $L^1$ norm represents the total number of particles and the norm with weight $x$ corresponds  to the ``total mass'' of the population.
Second, the definition of asynchronous exponential growth involves the bracket $\langle f,\phi\rangle$ which is implicitly assumed to be finite, and the largest Lebesgue space in which it can take place is then $L^1$ with the weight $\phi.$
The aim of the present paper is to obtain, under general conditions on the coefficients, uniform exponential convergence in $L^1$ spaces with weights as close as possible to $\phi.$
We extend in this sense some of the results of~\cite{MS} (see the comments below Theorem~\ref{maintheorem}).

For any positive weight function $\psi$ we denote by $L^1(\psi)$ the Lebesgue space $L^1((0,\infty);\psi\,\d x)$ endowed with the norm $\|f\|_{L^1(\psi)}:=\|f\psi\|_{L^1},$
and we simply use the shorthand $L^1_\alpha$ for the choice $\psi(x)=(1+x)^\alpha$ with $\alpha\in\R.$

\

We start by recalling an existence and uniqueness result for the Perron eigenvalue problem, obtained from~\cite[Theorem~1]{DG10}, \cite[Theorems~1.9~and~1.10]{BCG} and \cite[Theorem~2.1]{BG1}.
It ensures in particular that under our assumptions $L^1(\phi)\simeq L^1_1.$

\begin{theorem}\label{th:recall}
Assume that Hypotheses~{\bf (H$\tau$-H$B$-H$\wp$)} are satisfied.
There exist a unique solution (in the distributional sense) $(\lambda,G)\in\R\times L^1_0$ to the Perron eigenvalue problem~\eqref{eq:direct_Perron}
and a unique dual eigenfunction $\phi\in C^1(0,\infty)$ such that $(\lambda,\phi)$ satisfies~\eqref{eq:dual_Perron}.
Moreover $\lambda>0,$ $G\in L^1_\alpha$ for all $\alpha>0,$ and there exists a constant $C>0$ such that for all $x>0$
\[\frac1C(1+x)\leq\phi(x)\leq C(1+x).\]
\end{theorem}

\medskip

We are now in position to state the main results of the present paper, summarized in the following theorem.

\begin{theorem}
\label{maintheorem}
For any $\alpha\geq1$ and any $f^{\rm in}\in L^1_\alpha$ there exists a unique mild solution $f\in C([0,\infty),L^1_\alpha)$ to Equation~\eqref{eq:GF}.
If we assume additionally that 
\begin{enumerate}
\item[(i)] either $\wp$ is absolutely continuous with respect to the Lebesgue measure,
\item[(ii)] or $\supp\wp\subset[\epsilon,1-\epsilon]$ for some $\epsilon>0,$ and $\tau=const,$
\end{enumerate}
then for any $\alpha>\max(1,\a+2\gamma_1-2\gamma_0)$ there exist two constants $M,\sigma >0$ such that for all $f^{\rm in}\in L^1_\alpha$ and all $t\geq0$
\[\left\| f(t,\cdot)\e^{-\lambda t}-\langle f^{\rm in},\phi\rangle\,G\right\|_{L^{1}_{\alpha}}\leq
M\e^{-\sigma t}\left\|f^{\rm in}\right\|_{L^{1}_{\alpha}}.\]
\end{theorem}

\

\noindent Let us make some comments about the above results:
\begin{enumerate}
\item When $\a+2\gamma_1-2\gamma_0\leq1$ (for instance under condition \emph{(ii)} since in this case $\a=-\infty,$
or under condition~\emph{(i)} with $\gamma_1-\gamma_0\leq\frac12,$ as $\a$ is always nonpositive) the convergence holds for any $\alpha>1.$
In that event we get a close to optimal result since the $L^1_\alpha$ space can be chosen arbitrarily close to $L^1_1=L^1(\phi).$
The question whether it can be extended to $L^1(\phi)$ is still open.
A negative answer is given by~\cite{BG1} when $B$ is bounded (notice that in this case $\phi(x)\simeq (1+x)^k$ with $k<1$).
\item In~\cite{MS} the exponential convergence is proved for $\tau=const,$ $\gamma_0=\gamma_1\geq0$ and $\wp\in W^{1,1}(0,1)$ or $\wp=\delta_{\frac12},$
in the spaces $L^1_\alpha$ for all $\alpha>\alpha^*,$ where $\alpha^*\geq1$ is uniquely determined by $\wp_{\alpha^*}=B_0/B_1.$
We have generalized these assumptions, excepting the case $\gamma_0=\gamma_1=0$ which is not covered by~{\bf (H$B$)}.
Moreover we have strengthened the conclusion by extending the functional spaces for which it is valid.
Indeed, except for $B_0=B_1$ (implying that $B$ is exactly a power function for large sizes), we have $\alpha^*>1.$
\item For $\tau$ not satisfying~\eqref{as:tau0} we prove in Section~\ref{sec:Osgood} that the exponential convergence does not hold in $L^1(\phi).$
This ensures some kind of optimality for another result of~\cite{MS} which states that for $\tau(x)=x,$ $B(x)=x^{\gamma>0}$ and $\wp\in W^{1,1}(0,1),$ exponential convergence occurs in $L^1(x^{\alpha_1}+x^{\alpha_2})$ for any $0\leq \alpha_1<1<\alpha_2.$
Indeed these spaces are arbitrarily close to $L^1(x),$ which is equal to $L^1(\phi)$ when $\tau$ is linear.
\item We cannot expect a convergence result for $\tau(x)=x$ and $\wp$ general since it is known that for $\tau(x)=x$ and $\wp=\delta_{\frac12}$ the long time asymptotics of Equation~\eqref{eq:GF} consists in a periodic behavior~\cite{BDG,DoumicvanBrunt,vBALZ}.
\item Hypotheses~{\bf (H$\tau$)} and~{\bf (H$B$)} exclude the case $B=const$ and $\tau(x)=x,$ for which there is no Perron eigenfunction $G$ in $L^1_1$ and the behavior of the solutions to Equation~\eqref{eq:GF} is radically different from asynchronous exponential growth (see~\cite{BW16,DoumicEscobedo}).
\end{enumerate}

\medskip

The paper is structured as follows.
In Section~\ref{sec:semigroups} we prove the well-posedness of the growth-fragmentation equation and give some important properties of the associated semigroup.
In particular we obtain in Lemma~\ref{lm:moments_creation_T} new regularity estimates which are crucial for establishing the property of asynchronous exponential growth in Section~\ref{sec:asymptotic}.
In Section~\ref{sec:Osgood} we comment on the case when condition~\eqref{as:tau0} is not satisfied.


\section{Well-posedness of the Cauchy problem}
\label{sec:semigroups}

\subsection{Functional analytic setting}

First we look at the positive part $\cF_+$ of the fragmentation operator.
Since $B$ is a continuous function, the definition~\eqref{def:cF+} has a classical sense for $f$ continuous and compactly supported.
The continuous extension theorem ensures that it extends uniquely to a bounded positive operator from $L^1(\psi)$ to $L^1_\alpha,$
where $\psi(x)=(1+B(x))(1+x)^\alpha$ and $\alpha\in\R$ is such that $\wp_\alpha$ is finite.
From now on when talking about the operator $\mathcal F_+$ we mean this extension.

\begin{lemma}\label{lm:F+def}
Let $\alpha>\a$ and define $\psi(x)=(1+B(x))(1+x)^\alpha.$
There exists a unique bounded operator $\mathcal F_+:L^1(\psi)\to L^1_\alpha$ such that~\eqref{def:cF+} holds for any $f\in C_c(0,\infty).$
Additionally for all $f\in L^1(\psi)$
\[\|\mathcal F_+f\|_{L^1_\alpha}\leq \max(\wp_0,\wp_\alpha)\|f\|_{L^1(\psi)}.\]
\end{lemma}

\begin{proof}
It suffices to check that the claimed inequality is valid for all $f\in C_c(0,\infty).$
Let $f\in C_c(0,\infty)$ and $\alpha$ as in the lemma.
If $\alpha\geq0$ we have
\begin{align*}
\biggl\|\int_0^1 B\Bigl(\frac{\cdot}{z}\Bigr)f\Bigl(\frac{\cdot}{z}\Bigr)\frac{\wp(\d z)}{z}\biggr\|_{L^1_\alpha}&\leq\int_0^\infty\int_0^1 B\Bigl(\frac xz\Bigr)\Bigl|f\Bigl(\frac xz\Bigr)\Bigr|\frac{\wp(\d z)}{z}(1+x)^\alpha\d x\\
&\leq \int_0^\infty B(y)|f(y)|\int_0^1(1+zy)^{\alpha}\wp(\d z)\,\d y\\
&\leq \wp_0\int_0^\infty B(y)|f(y)|(1+y)^{\alpha}\,\d y=\wp_0\|Bf\|_{L^1_\alpha},
\end{align*}
and if $\alpha<0$
\begin{align*}
\biggl\|\int_0^1 B\Bigl(\frac{\cdot}{z}\Bigr)f\Bigl(\frac{\cdot}{z}\Bigr)\frac{\wp(\d z)}{z}\biggr\|_{L^1_\alpha}&\leq \int_0^\infty B(y)|f(y)|\int_0^1(1+zy)^{\alpha}\wp(\d z)\,\d y\\
&\leq \wp_{\alpha}\int_0^\infty B(y)|f(y)|(1+y)^{\alpha}\,\d y= \wp_\alpha\|Bf\|_{L^1_\alpha}.
\end{align*}
\end{proof}

Now we define on $L^1(\phi)$ the unbounded operator
$$
\cA_0 f:=-(\tau f)' -\lambda f -Bf,
$$
where $\lambda$ is the Perron eigenvalue defined in~\eqref{eq:direct_Perron}, with domain
$$
D\left(\cA_0\right):=\left\lbrace f\in L^1(\phi): (\tau f)' \in L^1(\phi),\ (\tau f)(0)=0,\ \mbox{and }Bf \in L^1(\phi)\right\rbrace.
$$
Lemma~\ref{lm:F+def} ensures that $\cF_+$ is well defined on $D(\cA_0)$ since $L^1(\phi)\simeq L^1_1$ and $\a<1.$
It will be considered as a perturbation of $\cA_0,$ with the same domain.

\medskip

With these definitions, the abstract Cauchy problem
\begin{equation}\label{eq:abstract_resc}\left\{\begin{array}{l}
\dfrac{\d}{\d t}\,g=\cA_0 g+\cF_+g
\vspace{3mm}\\
g(0)=f^{\rm in}
\end{array}\right.\end{equation}
corresponds to Equation~\eqref{eq:GF} rescaled by the exponential growth of parameter $\lambda.$
In other words $g$ is solution to~\eqref{eq:abstract_resc} if and only if $f=g\,\e^{\lambda t}$ is solution to~\eqref{eq:GF}.
We will first prove that $(\cA_0,D(\cA_0))$ generates a strongly continuous semigroup, also called $C_0$-semigroup, $(S_t)_{t\geq0}$ which admits a useful explicit formulation.
Then we will prove that the closure of $(\cA_0+\cF_+,D(\cA_0))$ generates a $C_0$-semigroup $(T_t)_{t\geq0},$ which satisfies a Duhamel formula.
Finally we check that $(T_t)_{t\geq0}$ is also a $C_0$-semigroup on $L^1_\alpha$ for any $\alpha\geq1.$
The semigroup $(T_t)_{t\geq0}$ yields the unique (mild) solution $g(t)=T_tf^{\rm in}$ to the abstract Cauchy problem~\eqref{eq:abstract_resc}.

\subsection{A $C_{0}$-semigroup for $\cA_0$}

\begin{proposition}\label{prop:gen_St}
The transport operator $(\cA_0,D(\cA_0))$ generates a positive contraction semigroup $(S_t)_{t\geq0}$ on $L^1(\phi).$
\end{proposition}

\begin{proof}
We prove that $\cA_0$ is dissipative and that $\mu-\cA_0$ is surjective for all $\mu>0.$
Then the Lumer-Philipps theorem (see~\cite[Theorem II.3.15]{EN} for instance) gives the result, since the density of $D(\cA_0)$ in $L^1(\phi)$ is clear.

\medskip

The dissipativity is due to the definition of $\phi,$
\[\cA_0^*\phi:=\tau\phi'-\lambda\phi-B\phi=-\cF_+^*\phi,\]
which ensures that for all $f\in D(\cA_0)$
\[\left\langle\cA_0 f, (\Sgn f)\phi\right\rangle=\langle\cA_0|f|,\phi\rangle=\langle|f|,\cA_0^*\phi\rangle=-\langle|f|,\cF_+^*\phi\rangle\leq0.\]

\medskip

For the surjectivity, let $\mu>0$ and $h\in L^1(\phi).$
The equation $(\mu-\cA_0)f=h$ is equivalent to solving the ordinary differential equation
\be\label{eq:fhmu}(\tau f)'(x)+(\lambda+\mu) f(x)+B(x)f(x)=h(x),\qquad x>0,\ee
with the initial condition $(\tau f)(0)=0.$
We obtain
\be\label{sol:fhmu}\tau(x)f(x)=\int_0^x \e^{-\int_y^x\frac{\mu+\lambda+B(z)}{\tau(z)}dz}h(y)\,dy.\ee
We need to verify that $f$ thus defined belongs to $D(\cA_0).$
Let's introduce
\[\Lambda(x):=\int_1^x\frac{\lambda+B(y)}{\tau(y)}\d y.\]
Since $\mu>0$ we get from~\eqref{sol:fhmu}
\begin{align*}
\int_0^\infty (\lambda+B(x))|f(x)|\phi(x)\,\d x&\leq\int_0^\infty\Lambda'(x)\e^{-\Lambda(x)}\phi(x) \int_0^x|h(y)|\e^{\Lambda(y)}\d y\d x\\
&\leq\int_0^\infty |h(y)|\phi(y)\bigg[\underbrace{\frac{\e^{\Lambda(y)}}{\phi(y)}\int_y^\infty \Lambda'(x)\e^{-\Lambda(x)}\phi(x)\,\d x}_{:=\Psi(y)}\bigg]\d y.
\end{align*}
We are going to prove that $\Psi(y)$ is bounded on $(0,\infty).$
As it is a continuous function which is bounded at $y=0,$ it suffices to check that it is bounded at $+\infty.$
Using that $y\Lambda'(y)=\frac{y(\lambda+B(y))}{\tau(y)}\to+\infty$ when $y\to+\infty$ we have
\[\frac{\d}{\d y}\big(y\e^{-\Lambda(y)}\big)=\big(1-y\Lambda'(y)\big)\e^{-\Lambda(y)}\sim_{y\to+\infty}- y\Lambda'(y)\e^{-\Lambda(y)}\]
and we deduce from the l'H\^opital's rule that
\[\int_y^\infty\Lambda'(x)\e^{-\Lambda(x)}x\,\d x\sim_{y\to+\infty} y\e^{-\Lambda(y)}.\]\\
Using the estimate on $\phi$ in Theorem~\ref{th:recall} we get for $y\geq1$
\[\Psi(y)\leq 2C^2 y^{-1}\e^{\Lambda(y)}\int_y^\infty \Lambda'(x)\e^{-\Lambda(x)}x\,\d x\xrightarrow[y\to+\infty]{}2C^2.\]
So $\Psi$ is bounded on $(0,\infty)$ and this ensures that $f$ and $Bf$ belong to $L^1(\phi).$
By Equation~\eqref{eq:fhmu} we deduce that $(\tau f)'\in L^1(\phi)$ too, and ultimately $f\in\cD(\cA_0).$

\medskip

The positivity of the semigroup results from the positivity of the resolvent $(\mu-\cA_0)^{-1},$ which is clear in~\eqref{sol:fhmu}.
\end{proof}

\begin{remark}
In the above proof, we have shown that $B$ is $\cA_0$-bounded
\[\forall f\in D(\cA_0),\qquad\left\|Bf\right\|_{L^1(\phi)}\leq \|\Psi\|_\infty\left\|\cA_0 f\right\|_{L^1(\phi)}\]
and as a consequence
\[D(\cA_0)=\left\lbrace f\in L^1(\phi): \cA_0f \in L^1(\phi)\ \text{and}\ (\tau f)(0)=0\right\rbrace.\]
\end{remark}

\

The semigroup $(S_t)_{t\geq0}$ generated by $\cA_0$ yields the solutions of the abstract Cauchy problem
\begin{equation}
\label{PCauchy0}
\left\{
    \begin{array}{l}
    \dfrac{\d}{\d t}u=\cA_0 u
     \vspace{2mm}\\
    u(0) = f.
    \end{array}
\right.
\end{equation}
Using the method of characteristics for transport equations, we can give another formula for the solution which provides, by identification, an explicit expression of the semigroup $(S_t)_{t\geq0}.$
As under Hypothesis~{\bf(H$\tau$)} the growth rate $\tau$ is globally Lipschitz, the Cauchy-Lipschitz Theorem ensures that for any $x\geq0$ the ordinary differential equation
$$
\left\{
\begin{array}{l}
\partial_t X\left(t,x\right)=\tau\left(X\left(t,x\right)\right)
\vspace{2mm}\\
X\left(0,x\right)=x
\end{array}
\right.
$$
has a unique maximal solution defined on the interval $[t_*(x),+\infty),$
where $t_*(x)\in(-\infty,0]$ is the time needed to reach the boundary $x=0,$ {\it i.e.} $X(t_*(x),x)=0,$
given by
$t_*(x)=-\int_0^x\frac{\d y}{\tau(y)}.$
Notice that we have used Assumption~\eqref{as:tau0} to get that $t_*(x)>-\infty.$
It is a standard result about the flow of an ordinary differential equation with a $C^1$ vector field that for any $t\geq0$ the mapping
\[X(t,\cdot):(0,\infty)\to(X(t,0),\infty)\]
is a diffeomorphism and that
\[X(t,\cdot)^{-1}=X(-t,\cdot).\]
Additionally we have for all $x\geq0$
\begin{equation}\label{eq:Xbound}
x\leq X(t,x)\leq (1+x)\e^{\tau_1t}-1.
\end{equation}
We can define for any $t\geq0$ and any $x>X(t,0)$
\[J(t,x):=\partial_xX(-t,x)\]
which is useful to compute explicitly the solutions of~\eqref{PCauchy0}.

\begin{proposition}\label{prop:explicit_St}
The semigroup $(S_t)_{t\geq0}$ is explicitly given by
\[S_{t}f(x) = \left\{\begin{array}{ll}
0&\text{if}\ 0< x\leq X(t,0),\\
f\left(X(-t,x)\right)J(t,x)\,\e^{-\int_{0}^{t}B\left(X(-s,x)\right)\d s}\e^{-\lambda t}\qquad&\text{if}\ x>X(t,0).
\end{array}\right.\]
\end{proposition}

\begin{proof}
For any $t\geq0,$ the operator $\tilde S_t$ defined by
\[\tilde S_tf(x)=\left\{\begin{array}{ll}
0&\text{if}\ 0<x\leq X(t,0),\\
f\left(X(-t,x)\right)J(t,x)\,\e^{-\int_{0}^{t}B\left(X(-s,x)\right)\d s}\e^{-\lambda t}\qquad&\text{if}\ x>X(t,0).
\end{array}\right.\]
is bounded on $L^1(\phi)$ since using Theorem~\ref{th:recall} and~\eqref{eq:Xbound} we have
\begin{align*}
\|\tilde S_tf\|_{L^1(\phi)}&\leq\int_0^\infty |f(X(-t,x))|J(t,x)\phi(x)\,\d x\\
&\leq\int_0^\infty |f(y)|\phi(X(t,y))\,\d y \leq C^2\e^{\tau_1 t}\int_0^\infty |u(y)|\phi(y)\,dy.
\end{align*}
Additionally $S_tf=\tilde S_tf$ for all $f\in C^1_c(0,\infty)\subset D(\cA_0),$
because they are both the unique (classical) solution to the transport equation~\eqref{PCauchy0}.
Indeed for $f\in C^1_c(0,\infty)$
it is a classical result obtained via the method of characteristics for transport equations that $\tilde S_tf$ is the solution to equation~\eqref{PCauchy0}.
Yet it can also be checked by direct computations.
First remark that if we define
\[F(x):=\int_0^x\frac{\d y}{\tau(y)},\]
which has a sense because of~\eqref{as:tau0},
we have the explicit formula
\[X(t,x)=F^{-1}(F(x)+t).\]
From this we deduce
\[J(t,x)=\frac{\tau(X(-t,x))}{\tau(x)}\qquad\text{and}\qquad\int_0^tB(X(-s,x))\,\d s=\int_{X(-t,x)}^{x}\frac{B(y)}{\tau(y)}\,\d y.\]
It is easy to check that $t\mapsto \tilde S_tf$ given by
\[\tilde S_tf(x)=\left\{\begin{array}{ll}
0&\text{if}\ 0<x\leq X(t,0),\\
f(X(-t,x))\dfrac{\tau(X(-t,x))}{\tau(x)}\,\e^{-\int_{X(-t,x)}^{x}\frac{B(y)}{\tau(y)}\,\d y}\e^{-\lambda t}\qquad&\text{if}\ x>X(t,0).
\end{array}\right.\]
 lies in $C^1_c(0,\infty)\subset D(\cA_0),$ is continuously differentiable in $L^1(\phi),$ and that its derivative is equal to $\cA_0\tilde S_tf.$

We conclude by density of $C^1_c(0,\infty)$ in $L^1(\phi)$ that $S_t=\tilde S_t.$
\end{proof}

The operator $S_t$ has been defined in $L^1(\phi)\simeq L^1_1.$
But due to the explicit formulation in Proposition~\ref{prop:explicit_St} we easily see that $L^1_\alpha$ with $\alpha>1$ is invariant under $S_t.$
Additionally the following lemma ensures that it extends uniquely to a linear operator in $L^1_\alpha$ for any $\alpha<1,$ and that when $t>0$ it has a regularizing property (creation of moments).

\begin{lemma}\label{lm:moments_creation_S}
For any $\alpha<1$ the operator $S_t$ extends uniquely to a bounded positive operator in $L^1_\alpha.$
Moreover if $t>0$ then for any $\alpha\in\R$ and any $\beta>\alpha$ the operator $S_t$ (or its extension) maps $L^1_\alpha$ into $L^1_\beta.$
More precisley for any $\beta\geq\alpha$ there exists $C=C(\alpha,\beta)>0$ such that for all $t>0$ and all $f\in L^1_\alpha$
\[\|S_tf\|_{L^1_\beta}\leq Ct^{-(\beta-\alpha)/\gamma_0}\e^{\beta\tau_1t}\|f\|_{L^1_\alpha}.\]
\end{lemma}

\begin{proof}
We use that $X(s,x)\geq x$ for $x\geq0$ and Assumption~\eqref{as:B_infty} to obtain for $f\in C_c(0,\infty)$
\begin{align*}
\|S_tf\|_{L^1_\beta}&\leq\int_{X(t,0)}^\infty |f(X(-t,x))|J(t,x)\e^{-\int_0^t B(X(-s,x))\d s}\e^{-\lambda t}(1+x)^\beta\,\d x\\
&\leq\e^{-\lambda t}\int_0^\infty |f(x)|\e^{-\int_0^t B(X(t-s,x))\d s}(1+X(t,x))^\beta\,\d x\\
&\leq\e^{(\beta\tau_1-\lambda) t}\bigg[\int_0^{x_0} |f(x)|(1+x)^\beta \,\d x+\int_{x_0}^\infty |f(x)|\e^{-B_0 x^{\gamma_0}t}(1+x)^\beta\,\d x\bigg]\\
&\lesssim \e^{(\beta\tau_1-\lambda) t}\bigg[\int_0^{x_0} |f(x)|(1+x)^\alpha\,\d x+\int_{x_0}^\infty|f(x)|\e^{-B_0 x^{\gamma_0}t}x^{\beta-\alpha}(1+x)^\alpha\,\d x\bigg],
\end{align*}
where the symbol $\lesssim$ denotes $\leq const\times.$
The first part of the lemma (extension to $L^1_\alpha,$ $\alpha<1$) is obtained by taking $\beta=\alpha$ and using the density of $C_c(0,\infty)$ in $L^1_\alpha.$
The second part follows from the fact that for $t>0$ and $\beta>\alpha$
\[\sup_{x\geq0}\big(\e^{-B_0 x^{\gamma_0}t}x^{\beta-\alpha}\big)=\e^{(\beta-\alpha)/\gamma_0}\Big(\frac{\beta-\alpha}{\gamma_0B_0t}\Big)^{(\beta-\alpha)/\gamma_0}.\]

\end{proof}

\subsection{The perturbed semigroup}
We consider $\cF_+,$ with domain $D(\cA_0),$ as a perturbation of $\cA_0.$
Unfortunately, as noticed in~\cite{EPW}, the operator $(\cA_0+\cF_+,D(\cA_0))$ is not closed.
Yet it is dissipative.
Indeed the definition of $\phi$ yields for all $f\in D(\cA_0)$
\begin{align*}
\left\langle(\cA_0+\cF_+) f, (\Sgn f)\phi\right\rangle&=\langle\cA_0|f|+(\cF_+f)\Sgn f,\phi\rangle\\
&\leq\langle(\cA_0+\cF_+)|f|,\phi\rangle=\langle|f|,(\cA_0^*+\cF_+^*)\phi\rangle=0.
\end{align*}
This ensures that $(\cA_0+\cF_+,D(\cA_0))$  is closable and its closure $\overline{\cA_0+\cF_+}$ is again dissipative (see for instance~\cite[Proposition II.3.14]{EN}).
We set $\cA:=\overline{\cA_0+\cF_+}$ which is defined by
\begin{align*}
D(\cA)=\big\{f&\in L^1(\phi): \exists (f_n)_{n\in\bN}\subset D(\cA_0),\ \exists h\in L^1(\phi),\\
&\|f_n-f\|_{L^1(\phi)}\to0\ \text{and}\ \|(\cA_0+\cF_+)f_n-h\|_{L^1(\phi)}\to0\big\}
\end{align*}
and $\cA f=h$ for all $f\in D(\cA).$
The fact that $\cA_0+\cF_+$ is not closed means that $D(\cA_0)\varsubsetneq D(\cA)$ and it is due to the unboundedness of $B.$
The reason, well illustrated in~\cite{EPW}, is the existence of functions $f\in L^1(\phi)$ with $(\tau f)'\in L^1(\phi)$ and $(\tau f)(0)=0$ such that $Bf$ and $\cF_+f$ do not belong to $L^1(\phi),$
but due to compensation $\cF f=\cF_+f-Bf\in L^1(\phi).$
Such functions belonging to $D(\cA)\setminus D(\cA_0)$ cannot be compactly supported.
More precisely if we denote by $D(\cA)_c$ (resp. $D(\cA_0)_c$) the subspace of $D(\cA)$ (resp. $D(\cA_0)$) composed of functions with a compact support in $[0,+\infty),$
we have $D(\cA)_c=D(\cA_0)_c$ and $\cA f=\cA_0f+\cF_+f$ for all $f\in D(\cA)_c.$
Indeed if $f\in D(\cA)$ is such that $\supp f\subset[0,R]$ for some $R>0,$ we can find a sequence $(f_n)\subset C^1_c(0,R+1)$ and $h\in L^1(0,R+1)$ such that $\|f_n-f\|_{L^1(\phi)}\to0$ and $\|(\cA_0+\cF_+)f_n-h\|_{L^1(\phi)}\to0.$
We obtain from $\|f_n-f\|_{L^1(\phi)}\to0$ and the local boundedness of $B$ that $Bf_n\to Bf$ and $\cF_+f_n\to\cF_+ f$ in $L^1(\phi).$
Then using $\|(\cA_0+\cF_+)f_n-h\|_{L^1(\phi)}\to0$ we deduce that $(\tau f_n)'=\cF_+ f_n-Bf_n-\lambda f_n-(\cA_0+\cF_+)f_n\to \cF_+ f-Bf-\lambda f-h$ in $L^1(\phi),$
guaranteeing that $(\tau f)'\in L^1(\phi)$ and $(\tau f)(0)=0.$
So $f\in D(\cA_0)_c.$

\

We start by proving that $\cA$ generates a $C_0$-semigroup in $L^1(\phi).$
Then we verify that it is also a $C_0$-semigroup in $L^1_\alpha$ for all $\alpha>1.$
Finally we give some useful properties of this semigroup.

\begin{theorem}
\label{th:gen_Tt}
The unbounded operator $(\cA,D(\cA))$ generates a positive $C_0$-semigroup $\left(T_{t}\right)_{t\geq0}$ on $L^1(\phi),$
which is conservative in the sense that if $f\geq0$ then for any $t\geq0$
\begin{equation}\label{eq:mass_cons}
\|T_tf\|_{L^1(\phi)}=\|f\|_{L^1(\phi)}.
\end{equation}
\end{theorem}

\begin{proof}
Again we use the Lumer-Philipps theorem.
We have already seen that $(\cA,D(\cA))$ is dissipative, so it remains to check that the range of $\mu-\cA$ is dense in $L^1(\phi)$ for some $\mu>0.$
To do so we generalize the proof in~\cite{EPW}.
Let define the set of fast decreasing functions
\[\cS:=\{f\in L^1(\phi): \forall k\geq0,\,f(x)=O(x^{-k})\ \text{when}\ x\to+\infty\}\]
and denote by $\cS_+$ its positive cone.
We will prove that $\cS,$ which is dense in $L^1(\phi),$ is included in the range of $\mu-\cA$ for $\mu$ large enough.
First we need an invariance property of $\cS_+.$

\medskip
\noindent{\bf Step 1:} \emph{The set $\cS_+$ is invariant under $\cF_+$ and $(\mu-\cA_0)^{-1}$ for any $\mu>0.$}\\
Let $f\in \cS_+,$ $k\geq0$ and $\mu>0.$
The positivity of $f$ is clearly preserved by $\cF_+$ and $(\mu-\cA_0)^{-1}$ which are positive operators.
Let $x_1>\max(1,x_0)$ and $c_1>0$ such that $\forall x\geq x_1,\ f(x)\leq c_1x^{-k}.$
If $k\geq\gamma_1+2$ we have for all $x>x_1$
\[\cF_+f(x)=\int_0^1B\big(\frac xz\big)f\big(\frac xz\big)\frac{\wp(dz)}{z}
\leq B_1 c_1x^{\gamma_1-k}\int_0^1z^{k-\gamma_1-1}\wp(dz)\leq B_1 c_1x^{\gamma_1-k}\]
and this ensures that $\cF_+f\in\cS.$
For $(\mu-\cA_0)^{-1}$ we start from~\eqref{sol:fhmu} and similarly as in the proof of Proposition~\ref{prop:gen_St} we write that for all $x>x_1$
\begin{align*}
(\mu-\cA_0)^{-1}&f(x)\leq \frac{\e^{-\Lambda(x)}}{\tau(x)}\int_0^{x_1}\e^{\Lambda(y)}f(y)\,\d y
+c_1\frac{\e^{-\Lambda(x)}}{\tau(x)}\int_{x_1}^x\e^{\Lambda(y)}y^{-k}\,\d y\\
&\leq \frac{\e^{-\Lambda(x)}}{\tau(x)}\int_0^{x_1}\e^{\Lambda(y)}f(y)\,\d y+\frac{c_1\tau_1}{\lambda+B_0}\frac{\e^{-\Lambda(x)}}{\tau(x)}\int_{x_1}^x\Lambda'(y)\e^{\Lambda_0(y)}y^{1-\gamma_0-k}\,\d y.
\end{align*}
To estimate the last term we use L'H\^opital's rule which gives
\[\int_{x_1}^x\Lambda'(y)\e^{\Lambda(y)}y^{1-\gamma_0-k}\,\d y\sim_{x\to+\infty}
x^{1-\gamma_0-k}\e^{\Lambda(x)}.\]
Finally we get for all $k\geq0$
\[(\mu-\cA_0)^{-1}f(x)=O\Big(\frac{\e^{-\Lambda(x)}+x^{1-\gamma_0-k}}{\tau(x)}\Big)\qquad\text{when}\ x\to+\infty\]
and we deduce that $(\mu-\cA_0)^{-1}f\in \cS_+.$

\medskip
\noindent{\bf Step 2:} \emph{Density of the range.}\\
Define $k_B:=\lfloor\gamma_1\rfloor+2$ and let $h\in\cS_+.$
For $\mu>0$ (large) to be chosen later, set $f_0=(\mu-\cA_0)^{-1}h$ and define the sequence $f_n$ inductively by
\[f_{n+1}=f_0+(\mu-\cA_0)^{-1}\cF_+f_n.\]
Using that $\cF_+$ and $(\mu-\cA_0)^{-1}$ are positive we have $f_0\geq0,$ $f_1-f_0=(\mu-\cA_0)^{-1}\cF_+f_0\geq0,$ and by induction $f_{n+1}-f_n=(\mu-\cA_0)^{-1}\cF_+(f_n-f_{n-1})\geq0.$ 
Hence $f_{n+1}\geq f_n$ pointwise.
Due to the step 1 we also have $(f_n)_{n\in\bN}\subset\cS_+$
and for any $k\in\bN$ and any $n\geq1$ we can integrate the equation
\[x^k(\mu-\cA_0)f_n(x)=x^kh(x)+x^k\cF_+ f_{n-1}(x).\]
on $(0,\infty).$
We get
\begin{align*}
-\int_0^\infty kx^{k-1}\tau f_n+(\mu+\lambda)\int_0^\infty x^kf_n+\int_0^\infty x^kBf_n&=\int_0^\infty x^kh+\wp_k\int_0^\infty x^kBf_{n-1}\\
&\leq \int_0^\infty x^kh+\wp_k\int_0^\infty x^kBf_n,
\end{align*}
which gives
\[(\mu+\lambda)\int_0^\infty x^kf_n\leq\int_0^\infty x^kh+k\int_0^\infty x^{k-1}\tau f_n+(\wp_k-1)\int_0^\infty x^kBf_n.\]
Considering $k=0$ and $k=1$ we obtain
\begin{equation}\label{eq:un-bound}
(\mu+\lambda-\tau_1)\int_0^\infty(1+x)f_n\leq\int_0^\infty(1+x)h+(\wp_0-1)\int_0^\infty Bf_n
\end{equation}
and for $k\geq2,$ since $\wp_k<\wp_1=1,$
\[(\mu+\lambda-k\tau_1)\int_0^\infty x^kf_n\leq\int_0^\infty x^kh+k\tau_1\int_0^\infty x^{k-1}f_n.\]
If $\mu>1+2k_B\tau_1-\lambda$ it yields for any $2\leq k\leq k_B$
\[\int_0^\infty x^kf_n\leq\int_0^\infty x^kh+\int_0^\infty x^{k-1}f_n\]
which gives by induction
\[\int_0^\infty x^{k_B}f_n\leq \sum_{k=2}^{k_B}\int_0^\infty x^kh+\int_0^\infty xf_n\]
and then
\[\int_0^\infty Bf_n\leq B_1\int_0^\infty(1+x^{k_B})f_n\leq B_1\sum_{k=2}^{k_B}\int_0^\infty x^kh+B_1\int_0^\infty(1+x)f_n.\]
Coming back to~\eqref{eq:un-bound} we get
\[(\mu+\lambda-\tau_1-(\wp_0-1)B_1)\int_0^\infty(1+x)f_n\leq
\int_0^\infty(1+x)h+(\wp_0-1)B_1\sum_{k=2}^{k_B}\int_0^\infty x^kh.\]
Finally if we choose $\mu>\max\big(1+2k_B\tau_1-\lambda,\tau_1+(\wp_0-1)B_1-\lambda\big)$ we obtain that $x^kf_n$ is bounded in $L^1_1\simeq L^1(\phi)$ for any $0\leq k\leq k_B-1.$
In particular $f_n$ and $\cF_+f_n$ are bounded in $L^1(\phi).$
By the monotone convergence theorem we may deduce that $f_n\to f_\infty$ and $(\mu-\cA_0-\cF_+)f_n=h+\cF_+(f_{n-1}-f_n)\to h$ in $L^1(\phi)$ when $n\to\infty.$
Hence $f_\infty\in D(\cA)$ and $(\mu-\cA)f_\infty=h.$
Since $\cS=\cS_+-\cS_+$ we may conclude that the range of $\mu-\cA$ is dense in $L^1(\phi),$ and this completes the proof of the generation of $(T_t)_{t\geq0}.$

\medskip
\noindent{\bf Step 3:} \emph{Positivity and conservation.}\\
The positivity of the semigroup follows from the positivity of the resolvent of $\cA,$ which is a consequence of the non-negativity of $f_\infty.$
The conservation property is guaranteed by the identity
\[\langle(\cA_0+\cF_+) f, \phi\rangle=0\]
which is valid for any nonnegative $f\in D(\cA_0).$

\end{proof}

\begin{remark}
A positive contraction semigroup is sometimes called \emph{substochastic semigroup}.
If it additionally satisfies the mass-preservation $\|T_tf\|=\|f\|$ for any $f\geq0$ it is called \emph{stochastic semigroup}.
Notice that the condition \mbox{$\langle(\cA_0+\cF_+) f, \phi\rangle=0$} is not sufficient to guarantee the stochasticity of $(T_t)_{t\geq0}$ in general.
In our case it is true because the semigroup $(T_t)_{t\geq0}$ is generated by the closure of $\cA_0+\cF_+.$
Mention also that the stochasticity of a semigroup is related to the notion of \emph{honesty}.
We refer to~\cite{ALMK,Banasiak,Tyran-Kaminska} for more details on these notions.
\end{remark}

Another useful property of the semigroup $(T_t)_{t\geq0}$ is that, as for the semigroup $(S_t)_{t\geq0},$
if $\supp f\subset[0,R]$ for some $R>0$ then $\supp T_tf\subset[0,X(t,R)].$
This can be seen for instance by means of the Dyson-Phillips expansion.
The perturbed operator $\cA_0+\cF_+$ verifies the assumptions of Kato's theorem (see for instance~\cite{ALMK} for a recent development).
It ensures the existence of an extension of $(\cA_0+\cF_+,D(\cA_0))$ generating a $C_0$-semigroup of contractions which is additionally given by the Dyson-Phillips expansion series.
Since we have proved that the closure of $(\cA_0+\cF_+,D(\cA_0))$ generates a $C_0$-semigroup, the Kato extension is necessarily $(\cA,D(\cA))$ (see~\cite[Proposition 3.8]{Banasiak}) and the Dyson-Phillips series which is strongly convergent in $L^1(\phi)$ reads for any $t\geq0$
\[T_t=\sum_{n=0}^\infty T_t^{(n)},\]
where $T_t^{(0)}=S_t$ and $T_t^{(n+1)}=\int_0^tT_{t-s}^{(n)}\cF_+S_s\,\d s.$
We easily check by induction that if $\supp f\subset[0,R]$ then $\supp T_t^{(n)}f\subset[0,X(t,R)]$ for all $n\in\bN.$
The initialization follows from the explicit formulation of $S_t,$ and the heredity results from the implication $\supp f\subset[0,R]\implies\supp \cF_+f\subset[0,R]$ and the identity $X(t-s,X(s,R))=X(t,R).$

\

We have proved the well-posedness of Equation~\eqref{eq:GF} in $L^1(\phi)\simeq L^1_1.$
Now we consider $\alpha>1$ and we establish that the semigroup $(T_t)_{t\geq0}$ defined in Theorem~\ref{th:gen_Tt} is also a $C_0$-semigroup on $L^1_\alpha.$

\begin{lemma}
\label{lm:boundednesspsi}
For any $\alpha>1$ the space $L^1_\alpha$ is invariant under the semigroup $(T_t)_{t\geq0},$ and there exists $C> 0$ such that for all $t\geq0$
\[\|T_{t}\|_{\scrL(L^1_\alpha)} \leq C\,(1+ t).\]
Additionally $(T_t)_{t\geq0}$ is a $C_0$-semigroup on $L^1_\alpha.$
\end{lemma}

\begin{proof}
Let $\alpha>1$ and let $f$ be an integrable function with compact support.
By definition of a mild solution of the abstract Cauchy problem~\eqref{eq:abstract_resc} with initial data $|f|$ we have $\int_0^tT_s|f|\,\d s\in D(\cA)$ for any $t>0$ and
\[T_t|f|=|f|+\cA\int_0^t T_s|f|\,\d s.\]
Additionally, due to what we explained just before the lemma, the integral $\int_0^tT_s|f|\,\d s$ has a compact support so it belongs to $D(\cA)_c=D(\cA_0).$
By integration against $x^\alpha$ it follows
\[\langle T_{t} |f|,x^\alpha\rangle= \langle |f|,x^\alpha\rangle+\int_0^t\langle T_s |f|,(\cA_0^*+\cF_+^*)x^\alpha \rangle\,\d s.\]
Since $\wp_\alpha<1$ and because of the assumptions on $\tau$ and $B,$
there exists $R>0$ such that
\[\forall x\geq R,\qquad \alpha\tau(x)/x-\lambda+(\wp_\alpha-1)B(x)\leq0.\]
This ensures that for all $t\geq0$
\begin{align*}
\langle T_t |f|,(\cA_0^*+\cF_+^*)x^\alpha \rangle
&= \int_0^\infty T_t|f|(x)\bigl[\alpha\tau(x)/x-\lambda+(\wp_\alpha-1)B(x)\bigr]x^\alpha\,\d x\\
&\leq \int_0^R T_t |f|(x) \bigl[\alpha\tau(x)/x-\lambda+(\wp_\alpha-1)B(x)\bigr]x^\alpha\,\d x\\
&\lesssim \int_0^\infty T_t|f|(x)\phi(x)\,\d x= \|f\|_{L^1(\phi)}\lesssim \|f\|_{L^1_\alpha}.
\end{align*}
Using that $|T_tf|\leq T_t|f|$ by positivity of $T_t,$ we deduce that
\[\langle |T_{t} f|,x^\alpha\rangle\leq\langle T_{t} |f|,x^\alpha\rangle\lesssim\langle |f|,x^\alpha\rangle+t\|f\|_{L^1_\alpha}\leq(1+t)\|f\|_{L^1_\alpha}\]
and then, since $(1+x)^\alpha\lesssim\phi(x)+x^\alpha$ and $T_t$ is a contraction in $L^1(\phi),$
\[\|T_tf\|_{L^1_\alpha}\lesssim\|T_tf\|_{L^1(\phi)}+\langle |T_{t} f|,x^\alpha\rangle\lesssim(1+t)\|f\|_{L^1_\alpha}.\]
We conclude with the density of the compactly supported functions in $L^1_\alpha.$

\medskip

It remains to prove the strong continuity of $(T_t)_{t\geq0}$ in $L^1_\alpha.$
The convergence $T_tf\to f$ in $L^1_\alpha$ readily follows from the convergence in $L^1(\phi)$ if $f$ is compactly supported.
Then it can be extended to any $f\in L^1_\alpha$ by a density argument.

\end{proof}

Now we establish, through a duality argument, a new result of creation of moments which will be the key argument for obtaining the asynchronous exponential growth.

\begin{lemma}
\label{lm:moments_creation_T}
For all $t>0$ and all $\beta>\alpha>1,$ $T_t$ is a bounded linear operator from $L^1_\alpha$ into $L^1_\beta.$
More precisely for all $\delta<\alpha$ there exist two positive constants $a=a(\alpha,\beta)$ and $C=C(\alpha,\beta,\delta)$ such that for all $f\in L^1_\alpha$ and all $t>0$
\[\|T_tf\|_{L^1_\beta}\leq Ct^{-(\beta-\delta)/\gamma_0}\e^{at}\|f\|_{L^1_\alpha}.\]
\end{lemma}

\begin{proof}
Fix $\beta>\alpha>\delta>1$ and denote $\psi(x)=1+x^\beta$ and $\varphi(x)=1+x^\delta.$
We have
\begin{align*}
(\cA_0^*+\cF_+^*)\psi(x)&=\beta\tau(x)x^{\beta-1}-\lambda\psi(x)+(\wp_0-1)B(x)+(\wp^{}_{\!\beta}-1)B(x)x^\beta\\
&\leq(\beta\tau_1-\lambda)\psi(x)+\big((\wp_0-1)-(1-\wp^{}_{\!\beta})x^\beta\big)B(x).
\end{align*}
Setting $c_\beta:=\frac12(1-\wp^{}_{\!\beta})B_0$ we can find $R_\beta\geq\max(1,x_0)$ such that for all $R\geq R_\beta,$ $(\wp_0-1)-(1-\wp^{}_{\!\beta})R^\beta\leq0$ and $\alpha\tau_1-\lambda-(1-\wp^{}_{\!\beta})B_0R^{\gamma_0}\leq -c_\beta R^{\gamma_0}.$
Choose such a $R_\beta$ and consider $R\geq R_\beta.$
For $x\geq R$ we have $(\wp_0-1)-(1-\wp^{}_{\!\beta})x^\beta\leq0$ and then
\begin{align*}
(\cA_0^*+\cF_+^*)\psi(x)&\leq(\beta\tau_1-\lambda)\psi(x)+\big((\wp_0-1)-(1-\wp^{}_{\!\beta})x^\beta\big)B_0R^{\gamma_0}\\
&=(\beta\tau_1-\lambda-(1-\wp^{}_{\!\beta})B_0R^{\gamma_0})\psi(x)+(\wp_0-\wp^{}_{\!\beta})B_0R^{\gamma_0}\\
&\leq -c_\beta R^{\gamma_0}\psi(x)+\frac{(\wp_0-\wp^{}_{\!\beta})B_0}{\varphi(R_\beta)}R^{\gamma_0}\varphi(x).
\end{align*}
For $x\leq R$ we have, since $(\wp_0-1)-(1-\wp^{}_{\!\beta})x^\beta\leq0$ for $x$ between $R_\beta$ and $R,$
\begin{align*}
(\cA_0^*+\cF_+^*)&\psi(x)\leq\beta\tau_1\psi(x)+(\wp_0-1)B_1R_\beta^{\gamma_1}\\
&\hspace{4mm} = -c_\beta R^{\gamma_0}\psi(x)+(\beta\tau_1+c_\beta R^{\gamma_0})\psi(x)+(\wp_0-1)B_1R_\beta^{\gamma_1}\\
&\hspace{2mm}\leq -c_\beta R^{\gamma_0}\psi(x)+\Big[(\beta\tau_1+c_\beta R^{\gamma_0})\frac{\psi(x)}{\varphi(x)}+(\wp_0-1)B_1R_\beta^{\gamma_1}\Big]\varphi(x)\\
&\leq -c_\beta R^{\gamma_0}\psi(x)+\Big[2(\beta\tau_1+c_\beta R^{\gamma_0})(1+R)^{\beta-\delta}+(\wp_0-1)B_1R_\beta^{\gamma_1}\Big]\varphi(x).
\end{align*}
Finally there exists a constant $C_\beta>0,$ independent of $R,$ such that for all $x\geq0$
\[(\cA_0^*+\cF_+^*)\psi(x)\leq -c_\beta R^{\gamma_0}\psi(x)+C_\beta R^{\gamma_0+\beta-\delta}\varphi(x).\]
Let $f\in D(\cA)_c$ be nonnegative.
Injecting the above inequality in
\[\frac{\d}{\d t}\langle T_{t} f,\psi\rangle= \langle \cA T_t f,\psi\rangle=\langle T_t f,(\cA_0^*+\cF_+^*)\psi\rangle\]
we get for all $R\geq R_\beta$
\[\frac{\d}{\d t}\langle T_{t} f,\psi\rangle\leq C_\beta R^{\gamma_0+\beta-\delta}\langle T_tf,\varphi\rangle -c_\beta R^{\gamma_0}\langle T_t f,\psi\rangle.\]
Lemma~\ref{lm:boundednesspsi} providing the existence of $C_\delta>0$ such that $\langle T_tf,\varphi\rangle\leq C_\delta(1+t)\langle f,\varphi\rangle,$ we deduce by a Grönwall type argument that
\[\langle T_tf,\psi\rangle\leq \e^{-c_\beta R^{\gamma_0}t}\langle f,\psi\rangle+\frac{C_\beta C_\delta}{c_\beta}R^{\beta-\delta}(1+t)\langle f,\varphi\rangle.\]
Since this inequality is valid for all nonnegative $f\in D(\cA)_c,$ it is equivalent to say that for all $t,x\geq0$ and all $R\geq R_\alpha$
\begin{equation}\label{T*psi}
T_t^*\psi(x)\leq \e^{-c_\beta R^{\gamma_0}t}\psi(x)+\frac{C_\beta C_\delta}{c_\beta}R^{\beta-\delta}(1+t)\varphi(x),
\end{equation}
where $T_t^*$ is the dual operator of $T_t,$ which acts on the dual space $L^1(\psi)'=L^\infty(\psi):=\{\varphi:(0,\infty)\to\R\ \text{measurable},\ |\varphi|/\psi\ \text{is essentially bounded on}\ (0,\infty)\}.$
Considering $R=\big(\frac{\beta-\alpha}{c_\beta}\frac{\log x}{t}\big)^{1/\gamma_0}$ we get that for $x\geq \exp\big(\frac{c_\beta}{\beta-\alpha}R_\beta^{\gamma_0}t\big)$
\begin{align*}
T_t^*\psi(x)&\leq x^{\alpha-\beta}\psi(x)+\frac{C_\beta C_\delta}{c_\beta}\Big(\frac{\beta-\alpha}{c_\beta}\Big)^{1/\gamma_0}t^{(\delta-\beta)/\gamma_0}(\log x)^{(\beta-\delta)/\gamma_0}\varphi(x)\\
&\leq C_{\alpha,\beta,\delta}(1+t^{(\delta-\beta)/\gamma_0})(1+x)^\alpha,
\end{align*}
where $C_{\alpha,\beta,\delta}$ is a positive constant independent of $t$ and $x.$
For $x< \exp\big(\frac{c_\beta}{\beta-\alpha}R_\beta^{\gamma_0}t\big)$ we use~\eqref{T*psi} with $R=R_\beta$ to get
\begin{align*}
T_t^*\psi(x)&\leq\e^{-c_\beta R_\beta^{\gamma_0}t}\psi\big(\e^{\frac{c_\beta}{\beta-\alpha}R_\beta^{\gamma_0}t}\big)+\frac{C_\beta C_\delta}{c_\beta}R_\beta^{\beta-\alpha}\varphi\big(\e^{\frac{c_\beta}{\beta-\alpha}R_\beta^{\gamma_0}t}\big)\\
&\leq 1+\e^{\frac{\alpha}{\beta-\alpha}c_\beta R_\beta^{\gamma_0}t}+\frac{C_\beta C_\delta}{c_\beta}R_\beta^{\beta-\alpha}\big(1+\e^{\frac{\alpha}{\beta-\alpha}c_\beta R_\beta^{\gamma_0}t}\big)\\
&\leq\Big(1+\frac{C_\beta C_\delta}{c_\beta}R_\beta^{\beta-\alpha}\Big)\big(1+\e^{\frac{\alpha}{\beta-\alpha}c_\beta R_\beta^{\gamma_0}t}\big)(1+x)^\alpha.
\end{align*}
Finally there exist two positive constants $a=a(\alpha,\beta)$ and $C=C(\alpha,\beta,\delta)$ such that for all $t>0$ and all $x\geq0$
\[T_t^*\psi(x)\leq Ct^{(\delta-\beta)/\gamma_0}\e^{at}(1+x)^\alpha\]
and as a consequence for all $f\in L^1_\beta$
\[\|T_tf\|_{L^1_\beta}\lesssim\|T_tf\|_{L^1(\psi)}\leq\langle T_t|f|,\psi\rangle=\langle |f|,T_t^*\psi\rangle\leq Ct^{(\delta-\beta)/\gamma_0}\e^{at}\|f\|_{L^1_\alpha}.\]
Then we may extend this inequality to all $f\in L^1_\alpha$ by a truncation argument.
\end{proof}

In addition to the Dyson-Phillips expansion, Kato's theorem guarantees the validity of the Duhamel formula
\[T_tf=S_tf+\int_0^tT_{t-s}\cF_+S_sf\,\d s\]
for all $t\geq0$ and $f\in D(\cA_0).$
Such an equation proves very useful for investigating the long time behavior of the semigroup $(T_t)_{t\geq0}.$
However in our study it is better to use a slightly different one, given in Lemma~\ref{lm:Duhamel} below.
The reason is that the property of creation of moments $S_t(L^1_\alpha)\subset L^1_\beta$ $(t>0,\beta>\alpha)$ is valid for any $\alpha\in\R,$
while in the proof of $T_t(L^1_\alpha)\subset L^1_\beta$ we need that $\alpha>1.$

\begin{lemma}\label{lm:Duhamel}
Let $\alpha>\max(1,\a+\gamma_1-\gamma_0).$
For any $t>0$ the integral
$\int_0^tS_{t-s}\cF_+T_s\,\d s$
defines a bounded linear operator on $L^1_\alpha.$
Moreover the following Duhamel formula holds in $\scrL(L^1_\alpha):$
\[T_t=S_t+\int_0^tS_{t-s}\cF_+T_s\,\d s,\qquad t\geq0.\]
\end{lemma}

\begin{proof}
Fix $t>0,$ $\alpha>\max(1,\a+\gamma_1-\gamma_0),$ $\beta\in(\a,\alpha+\gamma_0-\gamma_1),$ and $\delta\in(\beta+\gamma_1-\gamma_0,\alpha).$
We use Lemmas~\ref{lm:F+def},~\ref{lm:moments_creation_S} and~\ref{lm:moments_creation_T} to get, uniformly in $s\in(0,t/2)$ and $f\in L^1_\alpha,$
\[\|S_{t-s}\cF_+T_sf\|_{L^1_\alpha}\lesssim\|\cF_+T_sf\|_{L^1_\beta}\lesssim\|T_sf\|_{L^1_{\beta+\gamma_1}}\lesssim s^{-(\beta+\gamma_1-\delta)/\gamma_0}\|f\|_{L^1_\alpha}.\]
Using Lemmas~\ref{lm:F+def} and~\ref{lm:moments_creation_T} we have uniformly in $s\in(t/2,t)$
\[\|S_{t-s}\cF_+T_sf\|_{L^1_\alpha}\leq\|\cF_+T_sf\|_{L^1_\alpha}\lesssim\|T_sf\|_{L^1_{\alpha+\gamma_1}}\lesssim \|f\|_{L^1_\alpha}.\]
Since $\delta>\beta+\gamma_1-\gamma_0$ we deduce that for any $f\in L^1_\alpha$ the function $s\mapsto\|S_{t-s}\cF_+T_sf\|_{L^1_\alpha}$ is integrable on $(0,t)$ and it ensures that the function $s\mapsto S_{t-s}\cF_+T_sf$ is (Bochner) integrable on $(0,t).$
We even proved that $s\mapsto\|S_{t-s}\cF_+T_s\|_{\scrL(L^1_\alpha)}$ is integrable on $(0,t)$ which, together with the triangular inequality $\|\int_0^tS_{t-s}\cF_+T_sf\,\d s\|_{L^1_\alpha}\leq\int_0^t\|S_{t-s}\cF_+T_sf\|_{L^1_\alpha}\,\d s,$ guarantees that the linear mapping
$f\mapsto \int_0^tS_{t-s}\cF_+T_sf\,\d s$ is a bounded operator on $L^1_\alpha.$

As a consequence it suffices to verify the Duhamel formula on a dense subspace of $L^1_\alpha.$
We use $D(\cA)_c=D(\cA_0)_c$ which is invariant under both semigroups $(S_t)_{t\geq0}$ and $(T_t)_{t\geq0}.$
For any $f\in D(\cA)_c=D(\cA_0)_c$ and $t>0$ we have
\begin{align*}
\frac{\d}{\d s}(S_{t-s}T_sf)&=-\cA_0S_{t-s}T_sf+S_{t-s}\cA T_sf\\
&=-\cA_0S_{t-s}T_sf+S_{t-s}(\cA_0+\cF_+) T_sf=S_{t-s}\cF_+T_sf.
\end{align*}
An integration between $0$ and $t$ yields the result.
\end{proof}

\section{Asymptotic behavior}
\label{sec:asymptotic}

\subsection{The essential spectrum}\label{ssec:essential_spectrum}

Recall that for a closed linear operator $A$ in a Banach space $\mathfrak X,$ the spectrum is defined by $\sigma(A):=\{\lambda\in\bC: A-\lambda \text{ is not bijective}\}$
and the spectral bound as $s(A):=\sup\{\Re\lambda: \lambda\in\sigma(A)\}.$
If $A$ is bounded, then the spectral radius $r(A):=\sup\{|\lambda|: \lambda\in\sigma(A)\}$ satisfies $r(A)\leq\|A\|_{\scrL(\mathfrak X)}.$
The operator $A-\lambda$ can be non bijective for various reasons and it is useful to define some subsets of the spectrum.
A notion which will play a key role in the proof of our main theorem is the essential spectrum.
There are several definitions of essential spectrum in the literature (see~\cite{GustafsonWeidmann}).
We will use the two following ones:
\[\sigma_{e1}(A):=\big\{\lambda\in\sigma(A): \rg(A-\lambda) \text{ is not closed or } \Ker(A-\lambda) \text{ is infinite dimensional}\big\}\]
and
\begin{align*}
\sigma_{e2}(A):=\big\{\lambda\in\sigma(A): & \rg(A-\lambda) \text{ is not closed, } \lambda \text{ is a limit point of } \sigma(A),\\
& \text{ or } \bigcup_{r\geq0}\ker\big((A-\lambda)^r\big) \text{ is infinite dimensional}\big\}.
\end{align*}
Accordingly we define the essential spectral radii $r_{ek}(T)=\sup\{|\lambda|: \lambda\in\sigma_{ek}(T)\}$ for $k=1,2.$

\medskip

The second definition is the one introduced by Browder in~\cite{Browder} and used by Webb in~\cite{Webb87} where an abstract theorem of asynchronous exponential growth is proved.
We will use the following statement which is readily deduced from Propositions~2.2, 2.3, 2.5 and Remarks 2.1 and 2.2 in~\cite{Webb87} (see also Corollary 4.2 in~\cite{MagalRuan}).

\begin{theorem}[\cite{Webb87}]
\label{th:abstraitissimo}
Let $(U_t)_{t\geq0}$ be a positive $C_0$-semigroup with infinitesimal generator $A$ in a Banach lattice $\mathfrak X.$
Assume that $r_{e2}(U_t)<r(U_t)$ for some (hence all) $t>0,$ and that there exists a strictly positive $\varphi\in\mathfrak X'$ such that for all $f\in\mathfrak X,$ $\langle \e^{-s(A)t}U_tf,\varphi\rangle$ is bounded in $t.$
Then there exists a positive finite rank operator $P$ in $\mathfrak X$ and two constants $M,\sigma>0$ such that $\|\e^{-s(A)t}U_t-P\|_{\scrL(\mathfrak X)}\leq Me^{-\sigma t}.$
\end{theorem}


The first definition of the essential spectrum is useful since it is proved in~\cite[Theorem 2]{Kato58} that it is invariant under strictly singular perturbation, and it is known from~\cite{Pelczynski} that in $L^1$ spaces weakly compact operators are strictly singular.
Combining these two results we deduce that if $A$ is a closed linear operator and $B$ a weakly compact operator in a $L^1$ space, then $\sigma_{e1}(A+B)=\sigma_{e1}(A).$

\medskip

Clearly we have $\sigma_{e1}(A)\subset\sigma_{e2}(A)$ but the two sets are not equal in general.
However it is proved in~\cite[Theorem 6.5]{LebowSchechter} (see also~\cite{Nussbaum}) that when $A$ is bounded the essential spectral radius is the same for both (and actually all standard) definitions,
{\it i.e.} $r_{e1}(A)=r_{e2}(A).$

\

\subsection{Proof of the asynchronous exponential growth}

This subsection is dedicated to the proof of the second part of Theorem~\ref{maintheorem} about the exponential convergence of $(T_t)_{t\geq0}$ to the rank-one projection $f\mapsto\langle f,\phi\rangle G.$
The idea is to apply Theorem~\ref{th:abstraitissimo}.

For the infinitesimal generator $\cA$ of the semigroup $(T_t)_{t\geq0}$ we have $s(\cA)=0.$
Indeed using Proposition~2.2 in~\cite{Webb87} one can define the growth bound $\omega_0(\cA)$ as $\omega_0(\cA):=\lim_{t\to\infty}\log(\|T_t\|)/t$ and
Lemma~\ref{lm:boundednesspsi} guarantees that $\omega_0(\cA)=0$ in $L^1_\alpha$ for any $\alpha>1$ (notice that by contractiveness of $T_t$ in $L^1(\phi)$ it is also true for $\alpha=1$).
Since $s(\cA)\leq\omega_0(\cA)$ and $0\in\sigma(\cA),$ we deduce that $s(\cA)=0.$
Hence if we can apply Theorem~\ref{th:abstraitissimo} we obtain the exponential convergence of $(T_t)_{t\geq0}$ to a positive finite rank projection $P.$
Then we easily deduce from the uniqueness of the Perron eigenfunction $G$ and the conservation law~\eqref{eq:mass_cons} that this projection is given by $Pf=\langle f,\phi\rangle G$ (see the proof of Corollary 5.4 in~\cite{BG1} for details).
It only remains to check the assumptions of Theorem~\ref{th:abstraitissimo}.

The conservation property~\eqref{eq:mass_cons} guarantees that $\langle T_tf,\phi\rangle$ is bounded in $t.$
The fact that the growth bound of $\cA$ is zero ensures that $r(T_t)=\e^{\omega_0(\cA)t}=1$ for all $t>0.$
The only missing assumption which has to be verified is that $r_{e2}(T_t)<1$ for some $t>0,$
meaning that the semigroup $(T_t)_{t\geq0}$ is quasi-compact (see~\cite{EN} for instance).
The end of the section is devoted to the proof of this property by using the  Duhamel formula in Lemma~\ref{lm:Duhamel}, which is recalled here
\[T_t=S_t+\int_0^tS_{t-s}\cF_+T_s\,\d s.\]
First we check that $r(S_t)<1$ for all $t>0$ (and any $\alpha\in\R$).
Then we prove that $\int_0^tS_{t-s}\cF_+T_s\,\d s$ is weakly compact in $L^1_\alpha$ when $\alpha>\max(1,\a+\gamma_1-\gamma_0).$
Using the properties of the essential spectral radius announced in Section~\ref{ssec:essential_spectrum} we deduce that
\[r_{e_2}(T_t)=r_{e1}(T_t)=r_{e1}(S_t)\leq r(S_t)<1.\]

\

The last inequality is easily obtained from the explicit formulation of $S_t.$

\begin{lemma}
For any $\alpha\in\R$ and $t>0$ one has $r(S_t)<1.$
\end{lemma}

\begin{proof}
Let $\alpha\in\R$ and $x_1\geq x_0$ such that $B_0x_1^{\gamma_0}>\alpha k\tau_1-\lambda.$
Consider $t_1>0$ defined by $X(t_1,0)= x_1.$
For all $t\geq t_1$ we have
\begin{align*}
\|S_tf\|_{L^1_\alpha}&\leq\int_0^\infty |f(x)|\e^{-\int_{t_1}^tB(X(s,x))ds}\e^{-\lambda t}(1+X(t,x))^\alpha\,\d x\\
&\leq\e^{-B_0x_1^{\gamma_0}(t-t_1)}\e^{-\lambda t+\alpha\tau_1t}\|f\|_{L^1_\alpha}.
\end{align*}
We deduce that $\omega_0(\cA_0)\leq\alpha\tau_1-\lambda-B_0x_1^{\gamma_0}<0$ and consequently \mbox{$r(S_t)=\e^{\omega_0(\cA_0)t}<1$} when $t>0.$
\end{proof}

\

Denote by $W(L^1_\alpha)$ the space of weakly compact operators in $L^1_\alpha,$
which is a (two-sided) ideal of the Banach algebra $\scrL(L^1_\alpha).$
For proving the weak compactness of $\int_0^tS_{t-s}\cF_+T_s\,\d s$ we iterate the Duhamel formula to get the identity
\[\int_0^tS_{t-s}\cF_+T_s\,\d s=\int_0^tS_{t-s}\cF_+S_s\,\d s+\int_0^t\bigg(\int_0^sS_{s-u}\cF_+S_u\,\d u\bigg) \cF_+T_{t-s}\,\d s,\]
and we prove that $\int_0^tS_{t-s}\cF_+S_s\,\d s$ and then $\int_0^t(\int_0^sS_{s-u}\cF_+S_u\,\d u) \cF_+T_{t-s}\,\d s$ belong to $W(L^1_\alpha).$
To do so we use that $W(L^1_\alpha)$ has the \emph{strong convex compactness property} (see~\cite{Schluchtermann,Weis}, 
or \cite{MK} for a direct proof in Lebesgue spaces).
This means that if a function \mbox{$U:(0,t)\to W(L^1_\alpha)$} is
\begin{itemize}
\item[-] strongly measurable, {\it i.e.} $\forall f\in L^1_\alpha$ the function $s\mapsto U(s)f$ is measurable,
\item[-] and strongly bounded, {\it i.e.} $\sup_{0<s<t}\|U(s)\|_{\scrL(L^1_\alpha)}<\infty,$
\end{itemize}
then $\int_0^tU(s)\,\d s\in W(L^1_\alpha).$
In our case unfortunately the strong boundedness assumption is not satisfied.
But it is easy to check that it can be replaced by the strong integrability assumption, which is that
\[s\mapsto\|U(s)\|_{\scrL(L^1_\alpha)}\ \text{is integrable on}\ (0,t).\]
It readily follows from the dominated convergence theorem together with the property that $W(L^1_\alpha)$ is closed in $\scrL(L^1_\alpha)$ (see for instance~\cite[Theorem II.C.6]{Wojtaszczyk}).
Notice that Schl\"uchtermann suggested in~\cite{Schluchtermann} that the assumption of strong boundedness should be replaced by the uniform integrability (which is even weaker than strong integrability)
\[\sup_{f\in L^1_\alpha}\int_0^t\|U(s)f\|_{L^1_\alpha}\d s<\infty\quad\text{and}\quad\lim_{|\Omega|\to0}\sup_{f\in L^1_\alpha}\int_\Omega\|U(s)f\|_{L^1_\alpha}\d s=0.\]

\

We start with a lemma.

\begin{lemma}\label{lm:S1}
Let $\alpha>\a+\gamma_1-\gamma_0$ and $\beta\in(\a,\infty)\cap[\alpha-\gamma_1,\infty).$
For any $t>0$ the integral $\int_0^tS_{t-s}\cF_+S_s\,\d s$ is a bounded linear operator from $L^1_\alpha$ to $L^1_\beta.$
More precisely there exists $C=C(\alpha,\beta)>0$ such that for all $t>0$ and all $f\in L^1_\alpha$
\[\bigg\|\int_0^tS_{t-s}\cF_+S_sf\,\d s\,\bigg\|_{L^1_\beta}\leq Ct^{1-(\beta+\gamma_1-\alpha)/\gamma_0}\e^{\beta\tau_1t}\|f\|_{L^1_\alpha}.\]
\end{lemma}

\begin{proof}
Let $\alpha$ and $\beta$ satisfy the assumptions of the lemma, and choose $\delta\in(\a,\alpha+\gamma_0-\gamma_1)\cap[\alpha-\gamma_1,\beta].$
Using Lemmas~\ref{lm:F+def} and~\ref{lm:moments_creation_S} we have uniformly in $0<s<t/2$ and $f\in L^1_\alpha$
\begin{align*}
\|S_{t-s}\cF_+S_sf\|_{L^1_\beta}&\lesssim \e^{\beta\tau_1(t-s)}(t-s)^{-(\beta-\delta)/\gamma_0}\|\cF_+S_sf\|_{L^1_\delta}\\
&\lesssim \e^{\beta\tau_1t}t^{-(\beta-\delta)/\gamma_0}\|S_sf\|_{L^1_{\delta+\gamma_1}}\\
&\lesssim \e^{\beta\tau_1t}t^{-(\beta-\delta)/\gamma_0}s^{-(\delta+\gamma_1-\alpha)/\gamma_0}\|f\|_{L^1_\alpha},
\end{align*}
and for $s\in(t/2,t)$
\[\|S_{t-s}\cF_+S_sf\|_{L^1_\beta}\!\lesssim\e^{\beta\tau_1t}\|\cF_+S_sf\|_{L^1_\beta}\!\lesssim\e^{\beta\tau_1t}\|S_sf\|_{L^1_{\beta+\gamma_1}}\!\lesssim \e^{\beta\tau_1t}t^{-(\beta+\gamma_1-\alpha)/\gamma_0}\|f\|_{L^1_\alpha}.\]
We deduce that for any $t>0$ the function $s\mapsto\|S_{t-s}\cF_+S_s\|_{\scrL(L^1_\alpha,L^1_\beta)}$ is integrable on $(0,t)$ and after integration we get
\[\int_0^t\|S_{t-s}\cF_+S_s\|_{\scrL(L^1_\alpha,L^1_\beta)}\,\d s\lesssim \e^{\beta\tau_1t}\,t^{1-(\beta+\gamma_1-\alpha)/\gamma_0}.\]
\end{proof}

\begin{proposition}\label{prop:S1W}
For each $\alpha>\max(1,\a+\gamma_1-\gamma_0)$ and each $t>0$ the operator $\int_0^tS_{t-s}\cF_+S_s\,\d s$ is weakly compact in $L^1_\alpha.$
\end{proposition}

\begin{proof}
Let $\alpha>\max(1,\a+\gamma_1-\gamma_0)$ and $t>0.$
We split the proof into two parts, corresponding to the two cases in Theorem~\ref{maintheorem}.
When $\wp$ is absolutely continuous we first prove that $\cF_+S_s$ is weakly compact for all $s>0$ and then use the strong convex compactness property.
For the case $\tau=const$ and $\supp\wp\subset[\ep,1-\ep]$ we prove directly the weak compactness of $\int_0^tS_{t-s}\cF_+S_s\,\d s.$

\smallskip
\noindent{\bf Case (i):} {\it $\wp$ absolutely continuous.}
From Lemmas~\ref{lm:F+def} and~\ref{lm:moments_creation_S} we easily get that for any $s>0$ and any $\beta>\alpha$ the operator $\cF_+S_s$ maps continuously $L^1_\alpha$ in $L^1_\beta.$
This guarantees the tightness of the image of the unit ball of $L^1_\alpha$ under $\cF_+S_s$ for any $s>0.$
Now we look at the uniform integrability.
Following the lines of the proof of Lemma~\ref{lm:moments_creation_S} we get for any $\Omega\subset(0,\infty),$ $s>0,$ and $f\in L^1_\alpha$
\begin{align*}
&\int_\Omega|\cF_+S_sf(x)|(1+x)^\alpha\,\d x\leq\int_0^1\int_{\Omega/z}B(x)|S_sf(x)|(1+zx)^\alpha\d x\,\wp(\d z)\\
&\quad\leq B_1\int_0^1\int_{X(-s,\frac{\Omega}{z})\cap(0,\infty)}|f(x)|\e^{-\int_0^sB(X(u,x))\d u}\e^{-\lambda s}(1+X(s,x))^{\alpha+\gamma_1}\d x\,\wp(\d z)\\
&\quad\lesssim s^{-\gamma_1/\gamma_0}\e^{(\alpha+\gamma_1)\tau_1s}\int_0^\infty\left[\int_{(0,1)\cap\frac{\Omega}{X(s,x)}}\wp(\d z)\right]\,|f(x)|(1+x)^{\alpha}\d x.
\end{align*}
The term between the brackets is small uniformly in $x>0$ when $|\Omega|$ is small because $\wp\in L^1(0,1)$ and $\big|\frac{\Omega}{X(s,x)}\big|=\frac{|\Omega|}{X(s,x)}\leq\frac{|\Omega|}{X(s,0)},$ with $X(s,0)>0$ due to Assumption~\eqref{as:tau0}.
This proves the uniform integrability condition and by the Dunford-Pettis theorem the operator $\cF_+S_s$ is weakly compact in $L^1_\alpha$ for any $s>0.$
Since $W(L^1_\alpha)$ is an ideal of $\scrL(L^1_\alpha)$ the operator $S_{t-s}\cF_+S_s$ is also weakly compact for each $s\in(0,t].$
Finally using the strong convex compactness property of $W(L^1_\alpha)$ we get the weak compactness of $\int_0^tS_{t-s}\cF_+S_s\,\d s$ in $L^1_\alpha.$
Clearly $s\mapsto S_{t-s}\cF_+S_s$ is strongly measurable due to the strong continuity of $(S_t)_{t\geq0},$
and the strong integrability readily follows from the inequality
\[\|S_{t-s}\cF_+S_sf\|_{L^1_\alpha}\lesssim s^{-(\delta+\gamma_1-\alpha)/\gamma_0}\|f\|_{L^1_\alpha}\]
that we established in the proof of Lemma~\ref{lm:S1}, with $\delta<\alpha+\gamma_0-\gamma_1.$

\medskip
\noindent{\bf Case (ii):} {\it$\tau=1$ and $\supp\wp\subset[\ep,1-\ep].$}
Lemma~\ref{lm:S1} guarantees that the integral $\int_0^tS_{t-s}\cF_+S_s\,\d s$ sends continuously $L^1_\alpha$ into $L^1_\beta$ for any $\beta>\alpha.$
Consequently the image of the unit ball of $L^1_\alpha$ under this operator is tight.
For the uniform integrability we write for $\Omega\subset(0,\infty)$ and $f\in L^1_\alpha$
\begin{align*}
&\int_\Omega\left|\int_0^tS_{t-s}\cF_+S_sf(x)\,\d s\right|(1+x)^\alpha\d x\\
&\quad\leq\int_\Omega\int_0^t|\cF_+S_sf(x-t+s)|\e^{-\int_0^{t-s}B(x-u)\d u}\,\d s\,(1+x)^\alpha\d x\\
&\quad=\int_\Omega\int_0^t \int_0^1B\Bigl(\frac {x-t+s}z\Bigr)\Big|S_sf\Bigl(\frac {x-t+s}z\Bigr)\Big|\frac{\wp(\d z)}{z}\,\e^{-\int_0^{t-s}B(x- u)\d u}\d s\,(1+x)^\alpha\d x\\
&\quad\leq B_1\int_\Omega\int_0^t \int_0^1\Bigl(1+\frac xz\Bigr)^{\gamma_1}\Big|f\Bigl(\frac {x-t+s}z-s\Bigr)\Big|\\
&\hspace{36mm}\e^{-\int_0^s B(\frac {x-t+s}z- u)\d u-\int_0^{t-s}B(x-u)\d u}\,\frac{\wp(\d z)}{z}\,\d s\,(1+x)^\alpha\d x.
\end{align*}
For $s\leq\frac t2$ we have
\[\e^{-\int_0^{t-s}B(x-u)\d u}\leq \e^{-\int_0^{\frac t2}B(x-u)\d u}\leq \mathbf1_{0< x< x_0}+\mathbf1_{x\geq x_0}\e^{-\frac t2 B_0(x-\frac t2)_+^{\gamma_0}}\]
and for $s\geq\frac t2$
\[\e^{-\int_0^s B(\frac {x-t+s}z- u)\d u}\leq \e^{-\int_0^{\frac t2} B(\frac {x-t+s}z- u)\d u}\leq \mathbf1_{0< x< x_0}+\mathbf1_{x\geq x_0}\e^{-\frac t2 B_0(x-t)_+^{\gamma_0}}.\]
The function $x\mapsto \big(1+\frac xz\big)^{\gamma_1}\big(\mathbf1_{0< x< x_0}+\mathbf1_{x\geq x_0}\e^{-\frac t2 B_0(x-t)_+^{\gamma_0}}\big)$ is clearly bounded on $(0,\infty),$ uniformly in $z\geq\ep,$
so we get
\begin{align*}
\int_\Omega\bigg|\int_0^tS_{t-s}\cF_+S_sf(x)&\,\d s\bigg|(1+x)^\alpha\d x\\
&\lesssim\int_0^t \int_0^1\int_\Omega\Big|f\Bigl(\frac {x-t+s}z-s\Bigr)\Big|(1+x)^\alpha\d x\,\frac{\wp(\d z)}{z}\,\d s.
\end{align*}
Set $\varphi(s,z,x)=\frac {x-t+s}z-s$ and do the change of variable $x\to y=\varphi(s,z,x).$\\
We obtain, since $\d y=\frac{\d x}{z}$ and $\varphi(s,z,\cdot)^{-1}(y)=z(y+s)+t-s\leq y+t,$
\begin{align*}
\int_\Omega\bigg|\int_0^tS_{t-s}&\cF_+S_sf(x)\,\d s\bigg|(1+x)^\alpha\d x\\
&\lesssim\int_0^t \int_0^1\int_{\varphi(s,z,\Omega)}|f(y)|(1+y+t)^\alpha\d y\,\wp(\d z)\,\d s\\
&=\int_0^\infty\int_0^t \int_0^1\mathbf1_{\varphi(s,z,\cdot)^{-1}(y)\in\Omega}\,\wp(\d z)\,\d s\,|f(y)|(1+y+t)^\alpha\d y\\
&=\int_0^\infty\int_0^1\int_0^t \mathbf1_{s\in\frac{1}{1-z}(t-zy-\Omega)}\,\d s\,\wp(\d z)\,|f(y)|(1+y+t)^\alpha\d y.
\end{align*}
Since $\big|\frac{1}{1-z}(t-zy-\Omega)\big|\leq\frac{|\Omega|}{\ep}$ for all $z\in\supp\wp\subset[\ep,1-\ep]$ we get
\[\int_\Omega\bigg|\int_0^tS_{t-s}\cF_+S_sf(x)\,\d s\bigg|(1+x)^\alpha\d x\lesssim|\Omega|\,\|f\|_{L^1_\alpha}\]
and the family $\big\{\int_0^tS_{t-s}\cF_+S_sf\,\d s : \|f\|_{L^1_\alpha}\leq1\big\}$ is uniformly integrable.
The Dunford-Pettis theorem yields the result.

\end{proof}

\begin{corollary}
For each $\alpha>\max(1,\a+2\gamma_1-2\gamma_0)$ and each $t>0$ the operator $\int_0^t\big(\int_0^sS_{s-u}\cF_+S_u\,\d u\big) \cF_+T_{t-s}\,\d s$ is weakly compact in $L^1_\alpha.$
\end{corollary}

\begin{proof}
It is a consequence of Proposition~\ref{prop:S1W} and the strong convex compactness property.
Fix $\alpha>\max(1,\a+\gamma_1-\gamma_0)$ and $t>0.$
Since $\int_0^sS_{s-u}\cF_+S_u\,\d u\in W(L^1_\alpha)$ and $\cF_+T_{t-s}\in\scrL(L^1_\alpha),$ the operator $\big(\int_0^sS_{s-u}\cF_+S_u\,\d u\big) \cF_+T_{t-s}$
is weakly compact for any $s\in(0,t).$
For checking the strong integrability on $(0,t)$ we use Lemmas~\ref{lm:F+def},~\ref{lm:moments_creation_T} and~\ref{lm:S1}.
Uniformly in $s\in(0,t/2)$ we have
\begin{align*}
\Big\|\Big(\int_0^sS_{s-u}\cF_+S_u\,\d u\Big) \cF_+T_{t-s}f\Big\|_{L^1_\alpha}\lesssim s\|\cF_+T_{t-s}f\|_{L^1_{\alpha+\gamma_1}}&\lesssim s\|T_{t-s}f\|_{L^1_{\alpha+2\gamma_1}}\\
&\lesssim s\|f\|_{L^1_\alpha}.
\end{align*}
Uniformly in $s\in(t/2,t)$ we have for any $\beta\in(\a+\gamma_1-\gamma_0,\alpha+\gamma_0-\gamma_1)$ and $\delta\in(\beta+\gamma_1-\gamma_0,\alpha)$
\begin{align*}
\Big\|\Big(\int_0^sS_{s-u}\cF_+S_u\,\d u\Big) \cF_+T_{t-s}f\Big\|_{L^1_\alpha}\lesssim\|\cF_+T_{t-s}f\|_{L^1_\beta}&\lesssim\|T_{t-s}f\|_{L^1_{\beta+\gamma_1}}\\
&\lesssim (t-s)^{-(\beta+\gamma_1-\delta)/\gamma_0}\|f\|_{L^1_\alpha}.
\end{align*}
So $s\mapsto\big\|\big(\int_0^sS_{s-u}\cF_+S_u\,\d u\big) \cF_+T_{t-s}f\big\|_{L^1_\alpha}$ is integrable on $(0,t)$ and we can apply the strong convex compactness property.

\end{proof}

\

\section{About the Osgood condition}
\label{sec:Osgood}

In this section we consider the case when the Osgood condition is satisfied
\be\label{as:Osgood}\lim_{x\to0}\int_x^1\frac{\d x}{\tau(x)}=+\infty,\ee
meaning that~\eqref{as:tau0} is not fulfilled.
Replacing Assumption~\eqref{as:tau0} by
\[\frac{B}{\tau}\in L^1(0,1),\quad\text{and}\quad\exists r\geq0,\ \sup_{0<x<1}x^{-r}\int_0^x\wp(dz)<+\infty\quad\text{and}\quad \frac{x^r}{\tau(x)}\in L^1(0,1)\]
still guarantees the existence and uniqueness of $(\lambda,G,\phi)$ (see~\cite{DG10}),
and $\phi$ still has a linear growth at $+\infty$~\cite[Theorem~1.9]{BCG} while $\phi(x)\sim const\times\e^{\Lambda(x)}$ when $x\to0$~\cite[Theorem~1.10]{BCG}.
Using these estimates on $\phi$ the proof of the generation of the semigroup $(T_t)_{t\geq0}$ can be readily adapted to the new assumptions.

Notice that in the particular case of the self-similar fragmentation, {\it i.e.} $\tau(x)=x$ and $B(x)=x^\gamma$ with $\gamma>0,$ a necessary and sufficient condition for the existence and uniqueness of $G$ is given in~\cite{BTK}
(and we easily check that $\lambda=1$ and $\phi(x)=\phi_0\, x$ verify~\eqref{eq:dual_Perron} for $\phi_0>0$ a suitable normalizing constant). 

\

The following result ensures that under the Osgood condition the convergence of $(T_t)_{t\geq0}$ to the projector $P:f\mapsto\langle f,\phi\rangle G$ cannot be uniform with respect the the initial distribution in $L^1(\phi).$
\begin{theorem}
Under Assumption~\eqref{as:Osgood} we have for all $t\geq0$
\[\| T_t-P\|_{\scrL(L^1(\phi))}=2.\]
\end{theorem}

\begin{proof}
Fix $t\geq0.$
First we have
\[\| T_t-P\|_{\scrL(L^1(\phi))}\leq\| T_t\|_{\scrL(L^1(\phi))}+\| P\|_{\scrL(L^1(\phi))}=2.\]
For the other inequality we consider the initial distribution $f_\eta(x):=\frac{1}{\eta\phi(x)}\mathbf1_{0<x<\eta}$ for $\eta>0$ small enough (to be determined later).
For any $R>0$ and any $\eta>0$ we have $Pf_\eta=G$ and
\begin{align*}
\|T_t f_\eta- G\|_{L^1(\phi)}&=\!\int_0^R\!|T_t f_\eta(x)-G(x)|\phi(x)\,\d x
+\!\int_R^\infty\!|T_t f_\eta(x)-G(x)|\phi(x)\,\d x\\
&\geq\int_R^\infty G(x)\phi(x)\,\d x-\int_R^\infty T_t f_\eta(x)\phi(x)\,\d x\\
&\hspace{20mm}+\int_0^RT_t f_\eta(x)\phi(x)\,\d x-\int_0^RG(x)\phi(x)\,\d x
\end{align*}
Let $\epsilon>0$ and $R>0$ such that $\int_R^\infty G\phi\geq1-\epsilon,$ and then $\int_0^R G\phi\leq\epsilon.$
Under assumption~\eqref{as:Osgood} the characteristic curves of the transport semigroup $(S_t)_{t\geq0}$ do not reach the boundary $0$ in finite time, {\it i.e.} $t_*(x)=-\infty$ for all $x>0.$
Consequently we can find $\eta$ small enough such that $\Supp T_tf_\eta\subset[0,R]$ and then
\[\|T_t f_\eta-G\|_{L^1_\alpha}\geq2(1-\epsilon).\]
\end{proof}

%
%

\end{document}